\documentclass[11pt]{amsart}

\pdfoutput=1

\usepackage{amsmath}
\usepackage{fullpage}
\usepackage{xspace}
\usepackage[psamsfonts]{amssymb}
\usepackage[latin1]{inputenc}
\usepackage{graphicx,color}
\usepackage{hyperref}
\usepackage{graphicx}
\usepackage{amsmath}%
\usepackage{amsthm}%
\usepackage{amscd}
\usepackage{amsfonts}%
\usepackage{amssymb}%
\usepackage{graphicx}

\usepackage{mathrsfs}

\usepackage{tikz}
\usetikzlibrary{matrix,arrows}

\usepackage{tikz-cd}

\newtheorem{theorem}{Theorem}[section]

\newtheorem{corollary}[theorem]{Corollary}

\newtheorem{lemma}[theorem]{Lemma}

\theoremstyle{remark}

\numberwithin{equation}{section}

\newcommand{\Acal}{\mathscr{A}}

\newcommand{\Fcal}{\mathscr{F}}

\newcommand{\Lcal}{\mathscr{L}}

\newcommand{\Ocal}{\mathscr{O}}

\newcommand{\Pro}{\mathbb{P}}

\newcommand{\C}{\mathbb{C}}

\newcommand{\Q}{\mathbb{Q}}
\newcommand{\R}{\mathbb{R}}

\newcommand{\A}{\mathbb{A}}

\newcommand{\Pic}{\mathrm{Pic}}

\newcommand{\vol}{\mathrm{vol}}

\makeatletter
\@namedef{subjclassname@2020}{%
  \textup{2020} Mathematics Subject Classification}
\makeatother

  \DeclareFontFamily{U}{wncy}{}
    \DeclareFontShape{U}{wncy}{m}{n}{<->wncyr10}{}
    \DeclareSymbolFont{mcy}{U}{wncy}{m}{n}
    \DeclareMathSymbol{\Sha}{\mathord}{mcy}{"58}

\begin{document}
\title[]{A criterion for non-density of integral points}

\author{Natalia Garcia-Fritz}
\address{ Departamento de Matem\'aticas,
Pontificia Universidad Cat\'olica de Chile.
Facultad de Matem\'aticas,
4860 Av.\ Vicu\~na Mackenna,
Macul, RM, Chile}
\email[N. Garcia-Fritz]{natalia.garcia@uc.cl}%

\author{Hector Pasten}
\address{ Departamento de Matem\'aticas,
Pontificia Universidad Cat\'olica de Chile.
Facultad de Matem\'aticas,
4860 Av.\ Vicu\~na Mackenna,
Macul, RM, Chile}
\email[H. Pasten]{hector.pasten@uc.cl}%

%%\thanks{}
\thanks{N.G.-F. was supported by ANID Fondecyt Regular grant 1211004 from Chile. H.P. was supported by ANID Fondecyt Regular grant 1230507 from Chile.}
\date{\today}
%%\dedicatory{}

\begin{abstract} We give a general criterion for Zariski degeneration of integral points in the complement of a divisor $D$ with $n$ components in a variety of dimension $n$ defined over $\mathbb{Q}$ or over a quadratic imaginary field. The key condition is that the intersection of the components of $D$ is not well-approximated by rational points, and we discuss several cases where this assumption is satisfied. We also prove a GCD bound for algebraic points in varieties, which can be of independent interest.
\end{abstract}

\maketitle

%%\tableofcontents

%%%%%%%%%%%%%%%%%%%%%%%%%%%%%%%%%%%%%%
%%%%%%%%%%%%%%%%%%%%%%%%%%%%%%%%%%%%%%
%%%%%%%%%%%%%%%%%%%%%%%%%%%%%%%%%%%%%%
%%%%%%%%%%%%%%%%%%%%%%%%%%%%%%%%%%%%%%
%%%%%%%%%%%%%%%%%%%%%%%%%%%%%%%%%%%%%%
%%%%%%%%%%%%%%%%%%%%%%%%%%%%%%%%%%%%%%

\section{Introduction} 

%%%%
%%%%
%%%%
\subsection{Non-density of integral points} It is a classical theorem of Siegel that if $X$ is a smooth projective curve over a number field $L$ and if $D>0$ is a reduced effective non-zero divisor such that $X-D$ is hyperbolic, then any set of $D$-integral points in $X$ is finite. This result gives a complete answer to the problem of density of integral points on affine curves. A celebrated result of Faltings \cite{Faltings1} removed the condition that $D$ be non-zero, thus covering the case of rational points in hyperbolic projective curves.

For a higher dimensional variety $X$ over $L$, the Lang--Vojta conjecture predicts that if $D\ge 0$ is an effective simple normal crossings divisor on $X$ such that $K_X+D$ is big ($K_X$ being a canonical divisor for $X$), then any set of $D$-integral points on $X$ is Zariski degenerate. When $X$ is a subvariety of an abelian variety and $D=0$ (the case of rational points) the problem is solved by results of Faltings \cite{Faltings2,Faltings3} building on ideas of Vojta \cite{VojtaCompact}; this is essentially all we know for $D=0$. On the other hand, a number results have been obtained when $D$ has several components by reduction to the case of integral points in semi-abelian varieties (cf. \cite{VojtaSA}) and by applications of Schmidt's subspace theorem (see for instance \cite{CZcurves, CZsurfaces, LevinAnn}). We can also mention that some constructions provide particular examples where Zariski degeneration of integral points has been established on varieties with a dense set of rational points and with $D$ irreducible, see for instance \cite{FaltingsGarden, Zirred}. 

Another technique, which allows to attack the case of $D$ with few components (although over specific number fields) comes from Runge's method. Before Siegel's theorem on integral points of curves, Runge \cite{Runge} proved some special cases over $\Q$ under the assumption that the divisor $D$ on the curve consisted of at least $2$ components defined over $\Q$. Runge's method has been extended to the higher dimensional setting by Levin \cite{Levin0,LevinRunge} (see also Le Fourn's extension \cite{LeFourn}), and over $\Q$ the main technical assumption is that the divisor $D$ on the variety consists of several components $D_j$ defined over $\Q$ having empty total intersection $\cap_jD_j$. This last assumption has been relaxed by Levin and Wang \cite{LW}, by requiring that $\cap_jD_j$ is finite and its geometric points are in fact rational over $\Q$, among other technical assumptions on the geometry of the divisors $D_j$. 

Our goal is to give a general criterion for Zariski degeneration of integral points over $\Q$ (and over quadratic imaginary fields) when $\cap_j D_j$ is a $0$-cycle that cannot be approximated too well by rational points. This turns out to be a rather mild assumption and we will discuss several cases where it is satisfied. 

%%%%
%%%%
%%%%
\subsection{Approximation of points} Let $L$ be a number field, let $X$ be a projective variety over $L$, and let $\Lcal$ be a line sheaf on $X$. Attached to this data there is a height function
$$
h(\Lcal, -):X(L)\to \R
$$
normalized by the factor $1/[L:\Q]$. If $D$ is a divisor on $X$ we simply write $h(D,-)$ instead of $h(\Ocal(D),-)$.  

Let $Y$ be an effective reduced subscheme on $X$ defined over $L$ and let $S$ be a finite set of places of $L$. One introduces the proximity function
$$
m_S(Y,-): (X-Y)(L)\to \R
$$
also normalized by the factor $1/[L:\Q]$. This is the sum of the local heights for places in $S$, see \cite{SilvermanHeight, SilvermanGCD} for the relevant definitions. The height and the proximity are well-defined up to adding a bounded function. 

For a line sheaf $\Lcal$ on $X$ we define the approximation coefficient of $Y$ relative to $\Lcal$ and $S$, denoted by $\tau_S(Y,\Lcal)$, as the infimum of all $\tau\ge 0$ such that the inequality
$$
m_S(Y,x) < \tau\cdot h(\Lcal, x) + O(1)
$$
holds for all rational points $x\in X(L)$ outside a Zariski closed set properly contained in $X$ (and possibly depending on $\tau$). As usual, $O(1)$ stands for a bounded function. If no such $\tau$ exists, then we set $\tau_S(Y,\Lcal)=\infty$. If $D$ is a divisor on $X$ we simply write $\tau_S(Y,D)$ instead of $\tau_S(Y,\Ocal(D))$. We will use the approximation coefficient in the case $Y$ is $0$-dimensional.

We will be mainly interested in the case when $S$ is chosen as the set of Archimedian places of $L$, in which case we simply write $m_\infty(Y,-)=m_S(Y,-)$ and $\tau_\infty(Y,\Lcal)=\tau_S(Y,\Lcal)$.
%%%%
%%%%
%%%%
\subsection{A general criterion}

Let us recall the notion of a set of integral points with respect to a divisor. If $D$ is a reduced effective divisor on $X$ defined over $L$, a set of rational points $\Phi\subseteq (X-D)(L)$ is said to be \emph{$(D,S)$-integral} if we have
$$
m_S(D,x)=h(D,x)+O(1)
$$
as $x$ varies in $\Phi$. When $S$ is chosen as the set of Archimedian places of $L$ we simply say that $\Phi$ is $D$-integral. For instance, if $X\subseteq \Pro^N$ is given by equations with coefficients in $O_L$ and $D=\{x_0=0\}\cap X$ is the reduced divisor on $X$ obtained by intersecting with the hyperplane at infinity, then (after identifying $\A^N$ with $\Pro^N-\{x_0=0\}$ in the standard way) the set of points in $\A^N\cap X$ with coordinates in $O_L$ is $D$-integral.

Our Main Criterion for non-density of integral points is the following:

\begin{theorem}[Main criterion]\label{ThmMC} Let $k$ be $\Q$ or a quadratic imaginary field. Let $X$ be a smooth projective variety over $k$ of dimension $n$. Let $D_1,...,D_n$ be reduced effective non-zero divisors over $k$ such that $D=\sum_{j=1}^n D_j$ is a simple normal crossings divisor, and such that for every component $Y$ defined over $k$ of the   intersection $\cap_{j=1}^n D_j$ (which is a $0$-cycle)  we have
\begin{equation}\label{EqnHypMC}
\tau_\infty(Y,D_j)<1, \quad \mbox{ for each }j=1,2,...,n.
\end{equation}
Then there is a properly contained Zariski closed set $Z_0\subseteq X$ such that the following holds:

Given any set of $D$-integral points $\Phi\subseteq (X-D)(k)$  we have
\begin{equation}\label{EqnMC}
\min_{1\le j\le n} h(D_j,x) = O(1)\mbox{ as }x\in \Phi-Z_0(k)\mbox{ varies}.
\end{equation}
In particular, under the same assumptions,
\begin{itemize}
\item[(i)] If each $D_j$ has positive Kodaira--Iitaka dimension, then every set of $D$-integral points is Zariski degenerate, and
\item[(ii)] If each $D_j$ is big, then there is a properly contained Zariski closed set $Z\subseteq X$ with the property that each set $\Phi$ of $D$-integral points is contained in $Z$ up to finitely many points.
\end{itemize}
\end{theorem}

This result is proved using a variation of Runge's method in Section \ref{SecMainProof}.

%%%%
%%%%
%%%%
\subsection{Applications} The Main Criterion \ref{ThmMC} can be applied whenever one has some control on the approximation properties of algebraic points in $X$. Here we will present a (non-exhaustive) list of cases where the Main Criterion is applicable. Some of these cases are a consequence of a general GCD bound for rational points that we prove in Section \ref{SecGCD} and which might be of independent interest ---in particular, using our GCD bound we will prove a special case of Vojta's main conjecture for rational homogeneous spaces of dimension up to $10$. 

In this section $k$ is either $\Q$ or a quadratic imaginary field, $X$ is a smooth projective variety over $k$ of dimension $n$, and $D_1,...,D_n$ are reduced effective non-zero divisors on $X$ defined over $k$ such that $D=\sum_{j=1}^n D_j$ is a simple normal crossings divisor.

Our first application builds on a well-known approximation theorem on abelian varieties that follows from Roth's theorem \cite{Serre}.

\begin{theorem}[Using differentials]\label{Thm1} Assume that $\dim H^0(X,\Omega^1_X)>0$. If each $D_j$ is big, then there is a properly contained Zariski closed set $Z\subseteq X$ with the property that each set $\Phi$ of $D$-integral points is contained in $Z$ up to finitely many points.
\end{theorem}

As a consequence, we will get the following result without the assumption on differentials.

\begin{theorem}[Small numerical rank]\label{Thm2} Let $r$ be the rank of the group generated by the classes of the divisors $D_j$ modulo numerical equivalence and assume that $r<n$. Suppose that each $D_j$ is big. Then every set of $D$-integral points is Zariski degenerate.
\end{theorem}

Note that the condition $r<n$ is satisfied, for instance, if the Picard rank of $X$ is less than $n$.

The next application uses our previously mentioned GCD bound for rational points.

\begin{theorem}[Using the volume of big divisors]\label{Thm3} Assume that each $D_j$ is big. Note that $Y=\cap_{j=1}^nD_j$ is either empty or zero dimensional; let $d\ge 0$ be the maximal number of geometric points in a component of $Y$ defined over $k$. If for each $j$ we have $\vol(D_j)>d$, then there is a properly contained Zariski closed set $Z\subseteq X$ with the property that each set $\Phi$ of $D$-integral points is contained in $Z$ up to finitely many points.
\end{theorem}

When $d=0$ we recover a special case of Levin's higher dimensional version of Runge's method \cite{LevinRunge}.
 
On the other hand, if each geometric point of $\cap_{j=1}^nD_j$ is $k$-rational and that the divisors $D_j$ are nef and big, the condition $\vol(D_j)>d$ becomes $D_j^n>1$. In this way we recover large part of the main arithmetic results of \cite{LW}.

Finally, we mention two applications that exploit the connection between Seshadri constants and Diophantine approximation discovered in \cite{McKinnonRoth}.

\begin{theorem}[Globally generated ample line sheaves]\label{Thm4} Assume that $n\ge 2$ and that for each $j$ we have $\Ocal(D_j)\simeq \Lcal_j^{\otimes m_j}$ for some ample globally generated line sheaves $\Lcal_j$ and integers $m_j\ge 2$.  Then there is a properly contained Zariski closed set $Z\subseteq X$ with the property that each set $\Phi$ of $D$-integral points is contained in $Z$ up to finitely many points.
\end{theorem}

\begin{theorem}[Large fundamental group]\label{Thm5} Assume that $X_\C$ has large algebraic fundamental group. If each $D_j$ is ample then there is a properly contained Zariski closed set $Z\subseteq X$ with the property that each set $\Phi$ of $D$-integral points is contained in $Z$ up to finitely many points.

Furthermore, if $X$ is a surface, then every set of $D$-integral points in $X$ is finite.
\end{theorem}

See \cite{Kollar} for the notion of ``large fundamental group''.

%%%%%%%%%%%%%%%%%%%%%%%%%%%%%%%%%%%%%%
%%%%%%%%%%%%%%%%%%%%%%%%%%%%%%%%%%%%%%
%%%%%%%%%%%%%%%%%%%%%%%%%%%%%%%%%%%%%%
%%%%%%%%%%%%%%%%%%%%%%%%%%%%%%%%%%%%%%
%%%%%%%%%%%%%%%%%%%%%%%%%%%%%%%%%%%%%%
%%%%%%%%%%%%%%%%%%%%%%%%%%%%%%%%%%%%%%

\section{Proof of the Main Criterion} \label{SecMainProof}

In this section we prove Theorem \ref{ThmMC}. For this we let $k$ be $\Q$ or a quadratic imaginary field. Let  $\infty$ denote the only place at infinity of $k$. We keep the notation of the statement of Theorem \ref{ThmMC}.

\begin{proof}[Proof of Theorem \ref{ThmMC}] We note that (i) and (ii) are formal consequences of the first part of the statement due to standard finiteness properties of heights. So we only need to prove \eqref{EqnMC}.

Recall our assumption \eqref{EqnHypMC}. Let $\tau$ be a constant such that for every component $Y$ of $\cap_j D_j$ defined over $k$, and each $j$ we have
$$
\tau_\infty(Y,D_j)< \tau < 1.
$$
Then there is $Z_0$ a Zariski closed set properly contained in $X$ (and containing the support of $D$) such that for each such $Y$ and $D_j$ we have 
$$
m_\infty(Y,x)< \tau \cdot h(D_j,x) +O(1)
$$
as $x$ varies in $(X-Z_0)(k)$. Take $\Phi\subseteq (X-Z_0)(k)$ a set of $D$-integral points. In what follows, the bounds for the error terms $O(1)$ depend on the choice of heights, proximity functions, and the set $\Phi$, but not on the particular point $x\in \Phi$ that we are considering. 

Let $x\in \Phi$. Since $\Phi$ is $D$-integral, we have
\begin{equation}\label{EqnProofMC1}
\sum_{j=1}^n h(D_j,x) = \sum_{j=1}^n m_\infty(D_j,x)+O(1).
\end{equation}
Let $W=\cap_{j=1}^n D_j$ and let $Y_1,...,Y_s$ be its components over $k$. Since there is only one place of $k$ under consideration, there an index $i_0$ (depending on $x$)  such that for all $i\ne i_0$ we have
$$
m_\infty(Y_i, x)=O(1).
$$
By definition of the proximity function to a subscheme \cite{SilvermanHeight,SilvermanGCD} and the normal crossings assumption, we have
$$
m_\infty(Y_{i_0},x)=m_\infty(W,x)+O(1)=\min_j m_\infty(D_j,x) + O(1).
$$
Let $j_0$ be an index for which $m_\infty(D_{j_0},x)=\min_j m_\infty(D_j,x)$ (it depends on $x$). We note that by the previous equalities es have
$$
\begin{aligned}
\sum_{j=1}^n m_\infty(D_j,x) &= m_\infty(D_{j_0},x) + \sum_{j\ne j_0} m_\infty(D_{j},x) \\
&= m_\infty(Y_{i_0},x)+ \sum_{j\ne j_0} m_\infty(D_{j},x) + O(1)\\
&= m_\infty(Y_{i_0},x)+ \sum_{j\ne j_0} h(D_{j},x) + O(1)
\end{aligned}
$$
where the last equality is by $D$-integrality. Together with \eqref{EqnProofMC1} we deduce
$$
h(D_{j_0},x) = m_\infty(Y_{i_0},x) + O(1).
$$
By the choice of $\tau$ and $Z_0$ we deduce
$$
h(D_{j_0},x) < \tau \cdot h(D_{j_0},x)+O(1).
$$
Since $\tau<1$ we obtain $h(D_{j_0},x) =O(1)$. By $D$-integrality we have
$$
\begin{aligned}
h(D_{j_0},x)&=m_\infty(D_{j_0},x) + O(1)\\
& = \min_j m_\infty(D_j,x) + O(1) = \min_j h(D_j,x) +O(1)
\end{aligned}
$$
and the result follows.
\end{proof}

%%%%%%%%%%%%%%%%%%%%%%%%%%%%%%%%%%%%%%
%%%%%%%%%%%%%%%%%%%%%%%%%%%%%%%%%%%%%%
%%%%%%%%%%%%%%%%%%%%%%%%%%%%%%%%%%%%%%
%%%%%%%%%%%%%%%%%%%%%%%%%%%%%%%%%%%%%%
%%%%%%%%%%%%%%%%%%%%%%%%%%%%%%%%%%%%%%
%%%%%%%%%%%%%%%%%%%%%%%%%%%%%%%%%%%%%%

\section{A general GCD bound} \label{SecGCD}

In this section $L$ is a number field (not necessarily $\Q$ or a quadratic imaginary field.) 

%%%%%%%%
%%%%%%%%
%%%%%%%%
%%%%%%%%
\subsection{The result}\label{SecGCD1} Let $X$ be a smooth projective variety over $L$ of dimension $n$. For a reduced effective $0$-cycle $Y$ on $X$ defined over $L$ one defines a height function
$$
h(Y,-): X(\overline{L})\to \R
$$
normalized to $\Q$. Namely, we let $\pi:X'\to X$ be the blow up of $X$ along $Y$ with exceptional divisor $E$, and let the height $h(Y,-)$ be induced by $h(E,-)$ on $X'$ ---the height on algebraic points relative to a line sheaf  or a divisor is standard, see for instance \cite{SilvermanHeight}. Let $d$ be the number of geometric points of $Y$. With this notation, we show

\begin{theorem}[GCD bound]\label{ThmGCD} Let $\Lcal$ be a big line sheaf on $X$ and let $\epsilon>0$. There is a properly contained Zariski closed set $Z\subseteq X$ such that 
$$
h(Y,x) < \left(\sqrt[n]{\frac{d}{\vol(\Lcal)}} + \epsilon\right) h(\Lcal,x)+O(1)
$$
as $x$ varies in $(X-Z)(\overline{L})$.
\end{theorem}

The height $h(Y,-)$ with respect to a subvariety of codimension bigger than $1$ is a generalization of the GCD of two numbers, and as discovered in \cite{SilvermanGCD} it is closely related to Vojta's conjecture, at least in the case of rational points. Indeed, the canonical class of the blow-up variety $X'$ is related to that of $X$ via the formula
$$
K_{X'} = (n-1)E + \pi^*K_X.
$$
Let $\Acal$ be a big line sheaf on $X$ and let $\epsilon>0$. Vojta's main conjecture \cite{VojtaThesis} applied to $X'$ then gives a properly contained Zariski closed subset $Z_0\subseteq X'$ depending on the previous data such that 
$$
(n-1)h(E,x) + h(\pi^*K_{X},x) < \epsilon \cdot h(\pi^*\Acal,x)
$$
as $x$ varies on $(X'-Z_0)(L)$. Therefore (for $n>1$) one gets
\begin{equation}\label{EqVojtaGCD}
h(Y,x) < \frac{1}{n-1} \cdot h(-K_{X},x) + \epsilon \cdot h(\Acal,x)
\end{equation}
as $x$ varies on $(X'-\pi(Z_0))(L)$. 

After the work of Bugeaud, Corvaja, and Zannier \cite{BCZ} a number of GCD bounds have been obtained for integral points with respect to a non-trivial divisor, using the Subspace Theorem, see for instance \cite{WangYasufuku}. However, Theorem \ref{ThmGCD} holds for algebraic points and no integrality condition is required. See also \cite{Grieve} for a GCD bound valid for rational points in a special case.

%%%%%%%%
%%%%%%%%
%%%%%%%%
%%%%%%%%
\subsection{Rational homogeneous spaces} The special case of rational homogeneous spaces leads to an interesting application of Theorem \ref{ThmGCD}. In this section we let $d=1$, that is, $Y$ consists of a single $L$-rational point of $X$.

\begin{theorem} Suppose that $X$ is a rational homogeneous space. Let $\epsilon>0$. Then there is a properly contained Zariski closed subset $Z_0\subseteq X$ such that
$$
h(Y,x) < \left(\frac{1}{2\sqrt[n]{n!}}+\epsilon\right)h(-K_X,x) + O(1)
$$
as $x$ varies over $(X-Z_0)(\overline{L})$.
\end{theorem}
\begin{proof} Since $X$ is Fano, one has that $-K_X$ is ample and thus $\vol(-K_X)=(-K_X)^n$. Then the result follows from Theorem \ref{ThmGCD} and the bounds in \cite{Snow}.
\end{proof}

\begin{corollary} Vojta's conjecture in the special case \eqref{EqVojtaGCD} holds when $Y$ is an $L$-rational point and $X$ is a rational homogeneous space of dimension $2\le n\le 10$. 
\end{corollary}
\begin{proof} This is because in that range of $n$ one has $2\sqrt[n]{n!}\ge n-1$.
\end{proof}

%%%%%%%%
%%%%%%%%
%%%%%%%%
%%%%%%%%
\subsection{Proof of the GCD bound} In this paragraph we prove our GCD bound for algebraic points. We keep the notation of Section \ref{SecGCD1}.

\begin{proof}[Proof of Theorem \ref{ThmGCD}] Let $0<\eta< \vol(\Lcal)$. By the definition of volume of a big line sheaf, there are arbitrarily large values of $s$ such that 
\begin{equation}\label{EqDimBd}
\dim H^0(X,\Lcal^{\otimes s})> \frac{\eta\cdot s^n}{n!}.
\end{equation}
The requirement that a non-zero global section of the sheaf $\Lcal^{\otimes s}$ vanishes along $Y$ with multiplicity at least $\mu$ at each geometric point imposes 
$$
d\cdot\binom{n+\mu}{n}
$$
linear conditions. Note that for large $\mu$  (and fixed $n=\dim X$ and $d$) the previous expression is asymptotic to $d\cdot \mu^n/n!$. 

Let $\delta>0$. There is a large enough choice of $s$ and $\mu$ such that 
$$
\frac{\eta\cdot s^n}{n!}>d\cdot \binom{n+\mu}{n}
\qquad\mbox{ and }\qquad
\frac{s}{\mu} < \sqrt[n]{\frac{d}{\eta}}+\delta
$$
in such a way that \eqref{EqDimBd} holds. Hence, there is a divisor $D$ on $X$ with $\Ocal(D)\simeq\Lcal^{\otimes s}$ and with multiplicity at least $\mu$ at each geometric point of $Y$. In particular, $\mu E$ is a component of $\pi^*D$. For algebraic points $y\in (X'-\pi^*D)(\overline{L})$ this gives
$$
\begin{aligned}
\mu\cdot h(Y,x)&=\mu\cdot h(E,y) + O(1) \le h(\pi^*D,y)+O(1)\\
&=h(D,x)  +O(1)= s\cdot h(\Lcal,x)+O(1)
\end{aligned}
$$
where $x=\pi(y)$. Since $\eta$ can be taken arbitrarily close to $\vol(\Lcal)$, this gives the result.
\end{proof}

%%%%%%%%%%%%%%%%%%%%%%%%%%%%%%%%%%%%%%
%%%%%%%%%%%%%%%%%%%%%%%%%%%%%%%%%%%%%%
%%%%%%%%%%%%%%%%%%%%%%%%%%%%%%%%%%%%%%
%%%%%%%%%%%%%%%%%%%%%%%%%%%%%%%%%%%%%%
%%%%%%%%%%%%%%%%%%%%%%%%%%%%%%%%%%%%%%
%%%%%%%%%%%%%%%%%%%%%%%%%%%%%%%%%%%%%%

\section{Applications} \label{SecExamples}

%%%%%%%%
%%%%%%%%
%%%%%%%%
%%%%%%%%
\subsection{Using differentials}

\begin{lemma}\label{Lemmaq}
Let $X$ be a variety over a number field $L$ and let $f:X\to A$ be a non-constant morphism from $X$ to an abelian variety $A$, defined over $L$. Let $\Lcal$ be a big line sheaf on $X$. Let $S$ be a finite set of places of $L$. Then for every reduced effective $0$-dimensional cycle $Y\subseteq X$ defined over $L$ we have $\tau_S(Y,\Lcal)=0$.
\end{lemma}
\begin{proof} Let $Y_0=f(Y)\subseteq A$ and let $\Fcal$ be an ample line sheaf on $A$. Let $\epsilon>0$. By the theorem in Section 7.3 of \cite{Serre} we have
$$
m_S(Y_0,x) < \epsilon \cdot h(\Fcal,x) + O(1)
$$
as $x$ varies in $(A-Y_0)(L)$. Then we have
$$
m_S(Y,x)\le m_S(f^*Y_0, x) < \epsilon\cdot h(f^*\Fcal, x) + O(1)
$$
as $x$ varies in $(X-f^*Y_0)(L)$. The result follows since there is $M>0$ and a properly contained Zariski closed subset $Z\subseteq X$ such that  $h(f^*\Fcal, x)\le M\cdot h(\Lcal,x) + O(1)$ as $x$ varies in $(X-Z)(L)$. 
\end{proof}

\begin{proof}[Proof of Theorem \ref{Thm1}] This follows from Lemma \ref{Lemmaq} by using the Albanese variety of $X$, and then applying Theorem \ref{ThmMC}.
\end{proof}

\begin{proof}[Proof of Theorem \ref{Thm2}] By Theorem \ref{Thm1} it suffices to assume that $\dim H^0(X,\Omega_X^1)=0$. In this case $\Pic^0(X)$ is trivial and numerical equivalence coincides with linear equivalence up to torsion. Hence, using Hilbert's 90 theorem and the assumption $r<n$, we see that there is non-constant a rational function $f:X\dasharrow \Pro^1$ defined over $k$ such that ${\rm div}(f)$ is supported in $D$. Any set of $D$-integral points in $(X-D)(k)$ is mapped by $f$ to a set of $([0]+[\infty])$-integral points in $\Pro^1$, but such sets are finite because $k$ is either $\Q$ or a quadratic imaginary field. Hence the result.
\end{proof}
%%%%%%%%
%%%%%%%%
%%%%%%%%
%%%%%%%%
\subsection{Using the GCD bound} 
\begin{proof}[Proof of theorem \ref{Thm3}] For any effective reduced $0$-dimensional cycle $Y\subseteq X$ defined over $k$ we have
$$
m_\infty(Y,x)\le h(Y,x)+O(1)
$$
as $x$ varies in $X(k)$, thanks to the local decomposition of heights. Now the result follows from Theorems \ref{ThmMC} and \ref{ThmGCD}.
\end{proof}
%%%%%%%%
%%%%%%%%
%%%%%%%%
%%%%%%%%
\subsection{Using Seshadri constants} For a smooth projective variety $X$ over $\C$, an ample line bundle $\Acal$ on $X$, and a point $P\in X(\C)$, there is the Seshadri constant $\varepsilon(P,\Acal)$ which measures positivity of $\Acal$ at $P$. See \cite{SeshadriPrimer} for the precise definitions and several fundamental properties. In particular, as $\Acal$ is ample,  $\varepsilon(P,\Acal)$ is a finite and positive real number.

In a remarkable work, McKinnon and Roth \cite{McKinnonRoth} discovered a deep connection between Seshadri constants and Diophantine approximation. This will be used to prove Theorems \ref{Thm4} and \ref{Thm5}.

\begin{proof}[Proof of Theorem \ref{Thm4}] By basic properties of Seshadri constants (cf. \cite{SeshadriPrimer}), for every geometric point $P\in X(\overline{k})$ and each $j$ we have 
$$
\epsilon(P,\Ocal(D_j))=m_j\cdot \epsilon(P,\Lcal_j)\ge m_j.
$$ 
Let $Y$ be a component over $k$ of $\cap_jD_j$. Note that a $k$-rational point in $X$ can approximate (with respect to the only archimedian place of $k$) at most one geometric point of $Y$. By Theorem 6.2 in \cite{McKinnonRoth} we deduce that 
$$
\tau_\infty(Y,D_j)\le \frac{n+1}{m_j\cdot n}\le \frac{n+1}{2n}<1
$$
since $n=\dim X\ge 2$. Now the result follows from Theorem \ref{ThmMC}.  
\end{proof}

We define the \'etale Seshadri constant by
$$
\hat{\varepsilon}(P,\Acal)=\sup\{\epsilon(Q,\pi^*\Acal) : \pi:X'\to X \mbox{ is finite \'etale and }\pi(Q)=P\}.
$$

\begin{lemma}\label{LemmaLarge} Let $L$ be a number field, $X$ a smooth projective variety over $L$ such that $X_\C$ has large algebraic fundamental group, $Y$ a reduced effective $0$-dimensional cycle on $X$ defined over $L$, $\Acal$ and ample line sheaf on $X$, and $S$ a finite set of places of $L$. Then we have
$$
\tau_S(Y,\Acal)=0.
$$  
\end{lemma}
\begin{proof} Let $P$ be a geometric point of $Y$, thus, $P\in X(\overline{L})$. By \cite{CerboCerbo} we have $\hat{\varepsilon}(P,\Acal)=\infty$. Then the result follows from Corollary 8.9 (applied to Theorem 6.2) in \cite{McKinnonRoth}.
\end{proof}

\begin{proof}[Proof of Theorem \ref{Thm5}] The first part of the result follows from Lemma \ref{LemmaLarge} and Theorem \ref{ThmMC}.  The assertion for surfaces is due to the fact that varieties with large algebraic fundamental group contain no rational curves, and then one applies Siegel's theorem.
\end{proof}

%%%%%%%%%%%%%%%%%%%%%%%%%%%%%%%%%%%%%%
%%%%%%%%%%%%%%%%%%%%%%%%%%%%%%%%%%%%%%
%%%%%%%%%%%%%%%%%%%%%%%%%%%%%%%%%%%%%%
%%%%%%%%%%%%%%%%%%%%%%%%%%%%%%%%%%%%%%
%%%%%%%%%%%%%%%%%%%%%%%%%%%%%%%%%%%%%%
%%%%%%%%%%%%%%%%%%%%%%%%%%%%%%%%%%%%%%
%%%%%%%%%%%%%%%%%%%%%%%%%%%%%%%%%%%%%%
%%%%%%%%%%%%%%%%%%%%%%%%%%%%%%%%%%%%%%
%%%%%%%%%%%%%%%%%%%%%%%%%%%%%%%%%%%%%%
%%%%%%%%%%%%%%%%%%%%%%%%%%%%%%%%%%%%%%
%%%%%%%%%%%%%%%%%%%%%%%%%%%%%%%%%%%%%%
%%%%%%%%%%%%%%%%%%%%%%%%%%%%%%%%%%%%%%

\section{Acknowledgments}

N.G.-F. was supported by ANID Fondecyt Regular grant 1211004 from Chile.

H.P. was supported by ANID Fondecyt Regular grant 1230507 from Chile.

We sincerely thank Aaron Levin for very useful feedback on an earlier version of this manuscript.
%%%%%%%%%%%%%%%%%%%%%%%%%%%%%%%%%%%%%%
%%%%%%%%%%%%%%%%%%%%%%%%%%%%%%%%%%%%%%
%%%%%%%%%%%%%%%%%%%%%%%%%%%%%%%%%%%%%%

\end{document}